\documentclass[11pt,reqno]{amsart}
\usepackage[utf8]{inputenc}
\usepackage[leqno]{amsmath}

\usepackage[a4paper, lmargin=0.10\paperwidth, rmargin=0.10\paperwidth, tmargin=0.10\paperheight, bmargin=0.10\paperheight]{geometry}

\usepackage{amssymb,amscd,amsmath,graphicx,enumerate}
\usepackage{hyperref}
\usepackage{graphicx}
\usepackage{pgf,tikz}
\usepackage{mathtools}

\theoremstyle{definition}
\newtheorem{Theorem}{Theorem}[section]
\newtheorem{lemma}{Lemma}[section]
\newtheorem{claim}{Claim}[section]

\newtheorem{definition}{Definition}[section]
\newtheorem{remark}{Remark}[section]
\newtheorem{notation}{Notation}[section]
\numberwithin{equation}{section}
\newtheorem{question}{Question}[section]
\newtheorem{con}{Conjecture}[section]

\DeclareMathOperator{\reg}{reg}

 \newcommand{\I}{\mathcal{I}}

\begin{document}

\title{Regularity of powers of $d$-sequence (parity) binomial edge ideals of unicycle graphs}

\author[Marie Amalore Nambi]{Marie Amalore Nambi}
\address{Department of Mathematics, Indian Institute of Technology Hyderabad, Kandi, Sangareddy - 502285}
\email{ma19resch11004@iith.ac.in}

\author[Neeraj Kumar]{Neeraj Kumar}
\email{neeraj@math.iith.ac.in}

\date{Submitted. March 10, 2023. Revised. December 11, 2023}

\begin{abstract} We classify all unicycle graphs whose edge-binomials form a $d$-sequence, particularly linear type binomial edge ideals. We also classify unicycle graphs whose parity edge-binomials form a $d$-sequence. We study the regularity of powers of (parity) binomial edge ideals of unicycle graphs generated by $d$-sequence (parity) edge-binomials. 
\end{abstract}

\maketitle

\section*{Introduction}

Let $G$ be a simple graph on $n$ vertices. Let $S=k[x_1,\ldots,x_n,y_1,\ldots,y_n]$ be a polynomial ring on $2n$ variables over an infinite field $k$. For a given graph $G$, one can associate binomial ideals in $S$; namely
\[
J_G = ( f_{ij} \; | \; \{i,j\} \in E(G) ) \text{ and } \mathcal{I}_G = ( g_{ij} \; | \; \{i,j\} \in E(G) )
\]
where $f_{ij}=x_iy_j-x_jy_i$ and  $g_{ij}=x_ix_j-y_jy_i$. The binomials $f_{ij}$ and $g_{ij}$ are called edge-binomial and parity edge-binomial, respectively. The binomial ideals $J_G$ and $\mathcal{I}_G$ are called \emph{binomial edge ideal} and \emph{parity binomial edge ideal} respectively. 

The notion of the binomial edge ideal was introduced by Herzog et al. in \cite{HH} and independently by Ohtani in \cite{O2011}. Herzog et al. have shown that the binomial edge ideal has natural connections to the study of conditional independence ideals that are suitable to model robustness in contexts of algebraic statistics (cf. \cite{HH}, \cite{JN}). One may also view the binomial edge ideal as a generalization of an ideal generated by a set of $2$-minors of $2 \times n$ matrix of indeterminates. The binomial edge ideal of path graphs coincides with the ideal of adjacent minors of a $2 \times n$ matrix of indeterminates. In (cf. \cite{DES}), Diaconis, Eisenbud, and Sturmfels studied the ideal generated by all the adjacent $2$-minors of a $2 \times n$ generic matrix. The authors also investigated the ideal generated by corner minors of a $2 \times n$ matrix of indeterminates which coincides with the binomial edge ideal of star graphs (cf. \cite{DES}). For a complete graph $K_n$, the homogeneous coordinate ring $S/J_G$ can be seen as a Segre variety given by the image of Segre product $\mathbb{P}^1 \times \mathbb{P}^{n-1}$ of projective spaces. The parity binomial edge ideal was introduced by Kahle et al. in \cite{TCT}. The authors studied primary decomposition, mesoprimary decomposition, Markov bases, and radicality of parity binomial edge ideals (cf. \cite{TCT}). In (\cite[Corollary 6.2]{B2018}), Bolognini et al. proved that if $G$ is a bipartite graph, then the binomial edge ideal $J_G$ coincides with the parity binomial edge ideal $\mathcal{I}_G$, Lovász–Saks–Schrijver ideal $L_G(2)$ (cf. \cite{LSS89}), and the Permanental edge ideal $\Pi_G$ (cf. \cite{HMSW}). More results on the equality of $J_G$ and $\mathcal{I}_G$ with $\Pi_G$, see \cite[Remark 3.4]{A2021LSS}.

The Rees algebra of an ideal encodes many asymptotic properties of that ideal. An ideal in a commutative ring is said to be of \emph{linear type} if their Rees and symmetric algebras are isomorphic; equivalently, the defining ideal of Rees algebra is generated by linear forms. The notion of $d$-sequence was introduced and initially studied by Huneke in \cite{H1980, H1982}. The author proved that an ideal generated by $d$-sequence in a commutative ring is of linear type (cf. \cite{H1980}). Villarreal characterized graphs for which edge ideals are of the linear type in \cite{V1995}, namely, edge ideals are of linear type if and only if the graph is a tree or has a unique cycle of odd length. Researchers have recently been interested in characterizing linear type (parity) binomial edge ideal. 

The binomial edge ideal $J_G$ is a complete intersection if and only if each component of $G$ is a path (cf. \cite{EHH,R2013}). Jayanthan et al. characterized graphs whose binomial edge ideal is an almost complete intersection in \cite[Theorems $4.3$, and $4.4$]{JAR}. One has the following strict inclusions: \emph{almost complete intersection ideal} $\implies$ \emph{d-sequence ideal} $\implies$ \emph{ideal of linear type} \cite{H1980,JAR} and \cite[page no. 341]{H1981SS}. Jayanthan et al. proposed the following conjecture.

\begin{con}
\cite[Conjecture 4.17]{JAR}\label{conj:LT} If the given graph is a tree or a unicyclic graph, then the binomial edge ideal is of linear type.
\end{con}

In \cite{A2022}, Kumar has shown that $J_G$ is of linear type for the closed graphs. In \cite{AN}, Amalore Nambi and Kumar have characterized all trees whose edge-binomials form a $d$-sequence; in particular, $J_G$ is of a linear type. In \cite{A2021LSS}, the author characterized graphs whose parity binomial edge ideals are complete intersection and almost complete intersection. By $\mathcal{T}_m$, we denote the class of trees on $n$ vertices having degree sequences $(m,2,\dots, 1)$ or $(m,1,\dots, 1)$, where $m \geq 2$. In Section \ref{sec.d-seq}, we classify unicyclic graphs whose (parity) edge-binomials form a $d$-sequence. We summarize the results below. 

\vspace{1mm}

\begin{Theorem} \label{thmU.1}
Let $S=k[x_1,\ldots,x_n,y_1,\ldots,y_n]$ be a polynomial ring. Let $G$ be a connected unicyclic graph and $H \in \mathcal{T}_m$ be a tree. Then the edge-binomials of $G$ form a $d$-sequence in $S$ if and only if the corresponding unicyclic graph has the following form
\begin{enumerate}
    \item[(a)] $G$ is obtained by adding an edge between a pendant vertex of $H$ and the center of $H$;
    \item[(b)] $G$ is obtained by adding an edge between the center of $H$ and an internal vertex of $H$;
    \item[(c)] $G=C_n$, where $n \geq 3$.
\end{enumerate}
Moreover, if char$(k) \neq 2$, and $G$ has one of the above forms, then the parity edge-binomials form a $d$-sequence in $S$. 
\end{Theorem}

The above theorem supports Conjecture \ref{conj:LT}. Moreover, we expect that parity binomial edge ideals of trees and unicyclic graphs are of linear type.

\vspace{1mm}

From a combinatorial commutative algebra point of view, one tries to understand the Castelnuovo-Mumford regularity of (parity) binomial edge ideal of a graph via combinatorial data of that graph. Matsuda and Murai obtained the bounds for the regularity of binomial edge ideal of graphs in terms of the number of the longest induced path and the number of vertices of that graph, \cite{MM}. Several authors investigated the regularity of binomial edge ideal of various class graphs (see \cite{JNR, CR20, MM, S21} for a partial list). Kumar \cite{A2021reg} obtained the lower bound for the regularity of parity binomial edge ideal of connected graphs. We have computed the regularity of $\I_G$ of graphs whose parity edge-binomials form a $d$-sequence, in Lemma \ref{lemma5.PU1} and Lemma \ref{lemma5.PU2}. In Section \ref{sec.reg}, we obtain the regularity of the product of the parity binomial edge ideal of disjoint union of paths and the binomial edge ideal of a complete graph in Theorem \ref{Thm3.parityIJ}.

\vspace{1mm}

Our primary interest is to study the regularity of powers of (parity) binomial edge ideals. Cutkosky, Herzog, and Trung (and independently Kodiyalam) proved that for any homogeneous ideal $I$, the regularity of $I^s$ is asymptotically a linear function in $s$ (cf. \cite{CHT1999, V2000}). Raghavan in \cite{R94} generalized the notion of $d$-sequence to quadratic sequence and studied the depth of powers of an ideal generated by quadratic sequences. Note that $d$-sequence implies quadratic sequences, but the converse need not be true \cite{R94}. Jayanthan et al. in \cite{JAR20} obtained the bounds for the regularity of powers of almost complete intersection binomial edge ideals using the quadratic sequence and related ideals approach. Furthermore, the authors provided the regularity of powers of binomial edge ideals of star graphs. Shen and Zhu in \cite{SZ} obtained explicit formulas for the regularity of powers of an almost complete intersection (parity) binomial edge ideals. Amalore Nambi and Kumar in \cite{AN} obtained explicit expression for the regularity of powers of binomial edge ideals of trees whose edge-binomial forms a $d$-sequence. Ene et al. in \cite{VRT2021} obtained the regularity of powers of binomial edge ideals of closed graphs. 

\vspace{1mm}

A \textit{clique} is a subset $U$ of $V(G)$ such that the induced subgraph $G[U]$ is a complete graph. If a vertex $v$ of $G$ is contained in only one maximal clique, it is called \emph{free vertex} of the graph; otherwise, \emph{internal vertex} of the graph. By $i(G)$, we denote the number of internal vertices of $G$. In Sections \ref{sec.BEI} and \ref{sec.PBEI}, we obtained a precise expression for the regularity of powers of (parity) binomial edge ideals of unicyclic graphs whose (parity) edge-binomial forms a $d$-sequence. We summarize the results below.

 \vspace{1mm}

\begin{Theorem} \label{U1.1}
Let $H \in \mathcal{T}_m$ be a tree, and $G$ be a unicyclic graph. Let $S=k[x_1,\ldots,x_n,y_1,\ldots,y_n]$ be a polynomial ring and let $J_G$ be the binomial edge ideal of $G$ and $\I_G$ be the parity binomial edge ideal of $G$. 
\begin{itemize}
\item[(i)] If $G$ is obtained by adding an edge between a pendant vertex of $H$ and the center of $H$, then for all $s\geq 1$, one has 

\begin{equation*}
\begin{split}
    \reg {S}/{J_{G}^s} &= \begin{cases}
    2s + i(G) -1 & \text{if girth}(G) = 3\\
    2s+ i(G) -3 & \text{if girth}(G) \geq 4,
\end{cases} \\
    \reg {S}/{\mathcal{I}_{G}^s} &= \begin{cases}
    2s + i(G) & \text{if girth}(G) = 3\\
    2s+ i(G) -2 & \text{if odd-girth}(G) \geq 5.
\end{cases} \\
\end{split}
\end{equation*}

\item[(ii)] If $G$ is obtained by adding an edge between the center of $H$ and an internal vertex of $H$, then one has 
\begin{equation*}
\begin{split}
    \reg {S}/{J_{G}^s} &= \begin{cases}
    2s + i(G) -1 & \text{if girth}(G) = 3\\
    2s+ i(G) -2 & \text{if girth}(G) \geq 4,
\end{cases} \\
\text{for all } s\geq 2,\\
    \reg {S}/{\mathcal{I}_{G}^s} &= \begin{cases}
    2s + i(G) -1 & \text{if girth}(G) = 3\\
    2s+ i(G) -2 & \text{if odd-girth}(G) \geq 5,
\end{cases} \\
\text{for all }  s\geq 1.
\end{split}   
\end{equation*}

\end{itemize}
\end{Theorem}

\vspace{2mm}

 \noindent {\bf Acknowledgement.} The first author is financially supported by the University Grant Commission, India. The second author is partially funded by MATRICS grant, project no. MTR/2020/000635, from Science and Engineering Research Board (SERB), India.

\section{Preliminaries}

In this section, we recall definitions from graph theory and commutative algebra. Throughout this article, $S$ denotes the polynomial ring $k[x_1,\ldots,x_n,y_1,\ldots,y_n]$ over a field $k$ unless otherwise stated.  

\vspace{2mm}

Let $M$ be a finitely generated graded $S$-module. Let $\mathbf{F}_{\bullet}$ be a minimal graded $S$-free resolution of $M$:
$$\mathbf{F}_{\bullet}:0 \longrightarrow \bigoplus_{j\in \mathbb{Z}} S(-p-j)^{\beta_{p,p+j}(M)}  \stackrel{\phi_{p}} \longrightarrow   \cdots \longrightarrow  \bigoplus_{j\in \mathbb{Z}} S(-j)^{\beta_{0,j}(M)} \stackrel{\phi_{0}} \longrightarrow M \longrightarrow  0.$$
Where, $S(-i-j)$ denotes the graded free module of rank $1$ obtained by shifting the degrees in $S$ by $i+j$, and $\beta_{i,i+j}(M)$ denotes the $(i,i+j)$-th graded Betti number of $M$ over $S$. From the minimal free resolution, one can obtain an important invariant called \emph{Castelnuovo-Mumford regularity} (or simply  \emph{regularity}) of $M$ over $S$, denoted by $\reg_S M$, is defined as
$$\reg_SM \coloneqq \max \{j \mid \beta_{i,i+j} \neq 0 \text{ for some } i\}. $$

For convenience, we shall use $\reg M$ instead of $\reg_SM$. We state a regularity lemma \cite[Corollary 20.19]{Eisenbud} that will be used throughout the article.

\begin{lemma}[Regularity lemma] \label{Lemma1.Reg}
Let $ 0 \rightarrow M_1 \rightarrow M_2 \rightarrow M_3   \rightarrow  0$ be a short exact sequence of finitely generated graded $S$-modules. Then the following holds.
\begin{enumerate}[(a)]
   \item $\reg{M_1} \leq \max \{\reg M_2, \reg M_3 +1\}$. The equality holds if $\reg M_2 \neq \reg M_3$.
   
   \item $\reg{M_2} \leq \max \{\reg M_1, \reg M_3\}$. The equality holds if $\reg M_1 \neq \reg M_3 +1$.
     
    \item $\reg{M_3} \leq \max \{\reg M_1 -1, \reg M_2\}$. The equality holds if $\reg M_1 \neq \reg M_2$. 
\end{enumerate}
\end{lemma}

Let $G$ be a finite simple graph on the vertex set $V(G)$ and edge set $E(G)$. For $U \subseteq V(G)$, an \emph{induced subgraph} on vertex set $U$ is denoted by $G[U]$, is for $i,j \in U$, $\{i,j\} \in E(G[U])$ if and only if $\{i,j\} \in E(G)$. A \emph{complete} graph on the vertex set $V(G) =[n]$ is denoted by $K_n$, is a graph with $\{i,j\} \in E(G)$ for all $i,j \in V(G)$. A graph $G$ is called a \emph{block graph} if every block of $G$ is a complete graph. A graph $G$ is said to be bipartite if there exists a bipartition $V(G) = V_1 \sqcup V_2$ such that no two vertices of $V_i$ are adjacent for each $i=1,2$. A \textit{star graph} on $[n+1]$ vertices is denoted by $K_{1,n}$, is the graph with one internal vertex and $n$ leaves. A \textit{cycle} on $[n]$ is denoted by $C_n$, is a graph with every vertex has degree $2$. A \textit{unicyclic graph} is a graph containing exactly one cycle as a subgraph.

\vspace{1mm}

\begin{notation} Let $G$ be a simple graph. For an edge $e'$ in $G$, $G\setminus e'$ is the graph on the vertex set $V(G)$ and edge set $E(G)\setminus {e'}$. An edge $e'$ is called a \emph{bridge} if $c(G) < c(G \setminus e')$, where $c(G)$ is the number of components of $G$. For a vertex $v$, 
$$N_{G}(v) = \{u \in V(G) \mid \{u,v\} \in E(G)\},$$ 
denotes the \emph{neighborhood} of $v$ in $G$. Let $e =\{i,j\} \notin E(G)$ be an edge in $G \cup \{e\}$. Then $G_e$ (cf. \cite[Definition 3.1]{MS}) is the graph on vertex set $V(G)$ and edge set $$E(G_e) = E(G) \cup \left\{ \{k,l\} : k,l \in N_{G}(i) \textnormal{ or } k,l \in N_{G}(j) \right\}.$$
\end{notation}

The following result provides equality of the binomial edge ideal and parity binomial edge ideal for a bipartite graph. 

\begin{remark}  \label{Rem1.phi} \cite[Corollary 6.2]{B2018}
Let $G$ be a bipartite graph with a partition $V(G)= V_1 \sqcup V_2$. Let $\phi : S \rightarrow S$ be a map defined by  
\begin{equation*}
    \phi(x_i) = \begin{cases}
                x_i \text{ if } i \in V_1 \\
                y_i \text{ if } i \in V_2
                \end{cases}
                \text{ and } 
    \phi(y_i) = \begin{cases}
                y_i \text{ if } i \in V_1 \\
                x_i \text{ if } i \in V_2.
                \end{cases}
\end{equation*}
Then $\phi$ is an isomorphism and $\phi(J_G) = \I_G$. We will reserve $\phi$ throughout the article.

\end{remark}

The following results describe the colon ideal operation on the binomial edge ideal due to Mohammadi and Sharifan and the parity binomial edge ideal due to Kumar. 

\begin{remark}\label{Rem1.MCI} 
Let $G$ be a simple graph.  
\begin{enumerate}[(a)]
 \item \cite[Theorem 3.4]{MS} Let $e=\{i,j\} \notin E(G)$ be a bridge in $G \cup \{e\}$. Then $J_G: f_e = J_{G_e}$.
\item \cite[Theorem 3.7]{MS} Let $e =\{i,j\} \notin E(G)$. Then 
\[
J_G : f_e = J_{G_e} + (g_{P,t} \mid \quad P:i,i_1,\ldots,i_s,j \text{ is a path between } i,j \text{ and } 0 \leq t \leq s),
\] 
where $g_{P,0}=x_{i_1}\ldots x_{i_s}$ and for each $1\leq t \leq s, g_{P,t}=y_{i_1}\ldots y_{i_t}x_{i_{t+1}}\ldots x_{i_s}$.
\item \cite[Lemma 3.3]{A2021LSS} Let $G$ be a non-bipartite graph. Assume that there exists an edge $e=\{u,v\}$ such that $G \setminus e$ is a bipartite graph. Then 
 \[\mathcal{I}_{G\setminus e}:g_e = \mathcal{I}_{G \setminus e} + ( f_{i,j} \mid \{i,j\} \in N_{G \setminus e}(u) \text{ or } \{i,j\} \in N_{G \setminus e}(v) ) = \phi(J_{(G\setminus e)_{e}}).\]
\end{enumerate}
\end{remark}

\begin{remark} \label{Rem1.pcolon}
    Let $G$ be a bipartite graph and $e=\{i,j\} \notin E(G)$ be a bridge in $G \cup \{e\}$  such that $V(G \cup e)= V_1 \sqcup V_2$. Then one has $\I_G: g_e = \I_{G} + ( f_{kl} \mid k,l \in N_{G}(i) \textnormal{ or } k,l \in N_{G}(j) )$.
\end{remark}
\begin{proof}
    The proof follows from Remarks \ref{Rem1.phi} and \ref{Rem1.MCI}(a).
\end{proof}

The following lemma follows from  \cite[Corollary 2.2]{HH}, \cite[Theorem 5.5]{TCT}  and \cite[Lemma 4.1]{JAR}.

\begin{lemma} \label{Lemma1.RD}
Let $G$ be  a simple graph and $g \in S$, then 
\begin{enumerate}[(a)]
    \item $J_G : g = J_G : g^n$ for any $n \geq 2$,
    \item $\I_G : g = \I_G : g^n$ for any $n \geq 2$, if char$(k) \neq 2$.
\end{enumerate}
\end{lemma}

\begin{remark}\label{Rem1.GUG}  \cite[Corollary 3.2]{JNR}
Let $G$ be a new graph obtained by gluing finitely many graphs at free vertices. Let $G=L_1 \cup\cdots\cup L_k$ be a graph satisfying the properties
\begin{enumerate}
  \item[(a)] For $i \neq j, $ if $L_i \cap L_j \neq  \emptyset $, then $L_i \cap L_j = \{v_{ij}\}$, for some vertex $v_{ij}$ which is free vertex in $L_i$ as well as $L_j$;
  \item[(b)] For distinct $i,j,k, L_i \cap L_j \cap L_k =  \emptyset. $
\end{enumerate}
Then $\reg S/J_G  = \sum_{i=1}^{k} \reg S/J_{L_{i}}$.
\end{remark}

\begin{remark}\label{rem1.Reg2}
\cite[Theorem 0.2]{AN} If $G\in \mathcal{T}_{m}$ (see, Notation \ref{defPK}), then $\reg S/J_{G}^s = 2s + i(G)-1$, for all $s \geq 1$ and $m \geq 2$.
\end{remark}

A \emph{flower graph} $F_{h,k}(v)$ (cf. \cite[Definition 3.1]{CR20}) is a connected block graph obtained by identifying $h$ copies of  $C_3$ and $k$ copies of $K_{1,3}$ with a common vertex $v$, where $v$ is one of the free vertices of $C_3$ and of $K_{1,3}$, and $cdeg(v) \geq 3$. $G$ is called flower-free if 
$G$ has no ﬂower graphs as induced subgraphs.

\begin{remark} \label{Rem.FF}
    \cite[Corollary 3.2]{CR20} Let $G$ be a connected block graph that does not have an isolated vertex. If $G$ is a flower-free graph, then $\reg S/J_G = i(G) + 1$.
\end{remark}

\begin{definition}\cite[Definition 1.1]{H1982} \label{Def-d-sequence} Let $R$ be a commutative ring. Set $a_0 = 0$. A sequence of elements $a_1,\ldots,a_m$ in $R$ is said to be a $d$-sequence if it satisfies the conditions: 
\begin{enumerate}[(a)]
    \item $a_1,\ldots,a_m$ is a minimal system of generators of the ideal $(a_1,\ldots,a_m)$;
    \item $( ( a_0,\ldots,a_i ) : a_{i+1}a_j) = (( a_0,\ldots,a_i ) : a_j)$ for all $0 \leq i \leq m-1$, and $j \geq i+1$.
\end{enumerate} 
\end{definition}

\begin{lemma} \label{Lemma1.PD}
\cite[Observation 2.4]{SZ}
Let $R$ be a commutative ring. Suppose that $a_1,\ldots,a_n$ form a $d$-sequence in $R$ and $I = ( a_1,\ldots,a_n )$ is the ideal in $R$. Set $a_0 = 0$. Then, one has 
\[(a_0,a_1,\ldots,a_{i-1})+I^{s}:a_i = (a_0,a_1,\ldots,a_{i-1}:a_i) + I^{s-1},\]
for $s\geq 1$ and $i = 1,\ldots,n$. 
\end{lemma}

\section{\textit{d}-sequence (parity) edge-binomials} \label{sec.d-seq}

\vspace{2mm}

In this section, we characterize connected unicyclic graphs whose edge-binomials form a $d$-sequence. In addition, we show that parity edge-binomials of these connected unicyclic graphs form a $d$-sequence. First, we set up notations that we use thought out this article.

\begin{notation}
    Let $a_1,\ldots,a_{n}$ be a sequence of (parity) edge-binomials of $G$. We denote $J_{i}$ for an ideal generated by $( a_1,\ldots,a_i )$ and $H_i$ for the graph associated to (parity) edge-binomials $a_1,\ldots,a_i$. We denote an edge  $\{\alpha_{a_{i}},\beta_{a_{i}}\} \in E(G)$ for an associated (parity) edge-binomial $a_i$.
\end{notation}

The following describes combinatorial characterization for edge-binomials of a unicyclic graph to form a $d$-sequence.

\begin{lemma} \label{Lemma2.an}
Let $G$ be a unicyclic graph on $[n]$. Assume that $a_1,\ldots,a_{n}$ form a  $d$-sequence edge-binomials of $G$. Then an edge associated with $a_n$ forms a cycle in $G$, i.e., the graph $H_{n-1}$ is a tree.
\end{lemma}
\begin{proof}
Let $i<n$ be the smallest integer such that $H_i$ has a cycle. Then from Remark \ref{Rem1.MCI} (a) and (b) it follows that $J_{i-1}:a_{i}a_{n} \neq J_{i-1}:a_{n}$, since the edge $e_{n}$ is a bridge in $H_{i-1} \cup e_n$, while  edge $e_i$ is not a bridge in $H_i$. This contradicts the hypothesis, thus $i=n$, as desired. 
\end{proof}

The authors in \cite{AN} proved the following lemma for trees. We prove the result for unicyclic graphs. 

\begin{lemma}\label{Lemma2.T}
Let $G$ be a unicyclic graph on $[n]$. Assume that $a_1,\ldots,a_{n}$ are  $d$-sequence edge-binomials of $G$, where $a_k$ corresponds to an edge $\{\alpha_{a_{k}},\beta_{a_{k}}\} \in E(G)$ for all $k$. 

If there exists a smallest integer $i$ such that $J_i:a_{i+1} \neq J_i$, then $\{\alpha_{a_{i+1}},\beta_{a_{i+1}}\} \cap \{\alpha_{a_{j}},\beta_{a_{j}}\} \neq \emptyset$, for all $j>i+1$. In particular, for all $j>i+1$, one has $\{\alpha_{a_{i+1}}\} \cap \{\alpha_{a_{j}},\beta_{a_{j}}\} \neq \emptyset$ or $\{\beta_{a_{i+1}}\} \cap \{\alpha_{a_{j}},\beta_{a_{j}}\} \neq \emptyset$.
\end{lemma}
\begin{proof}
Assume that $i<n-1$. From Lemma  \cite[Lemma 2.1]{AN} and \ref{Lemma2.an}, it is enough to prove that $\{\alpha_{a_{i+1}},\beta_{a_{i+1}}\} \cap \{\alpha_{a_{n}},\beta_{a_{n}}\} \neq \emptyset$. Suppose that $\{\alpha_{a_{i+1}},\beta_{a_{i+1}}\} \cap \{\alpha_{a_{n}},\beta_{a_{n}}\} = \emptyset$ and $i$ be the smallest integer such that  $J_i:a_{i+1} \neq J_i$. Then from Remark \ref{Rem1.MCI} one has $f_{kl} \in J_i:a_{i+1}$ such that $f_{kl} \notin J_i$, where $k,l \in N_{H_{i}}(\alpha_{a_{i+1}})$ or $k,l \in N_{H_{i}}(\beta_{a_{i+1}})$. If $k,l \in N_{H_{i}}(\alpha_{a_{i+1}})$ then clearly $k,l \notin N_{H_{i}}(\beta_{a_{i+1}})$, otherwise $G$ has two cycles. So one can further assume that $k,l \in N_{H_{i}}(\alpha_{a_{i+1}})$. We have $J_i:a_{i+1}a_n= J_i:a_n$ by hypothesis. But, if $f_{kl} \in J_i:a_{n}$ then $k,l \in N_{H_{i}}(\alpha_{a_{n}})$ or $k,l \in N_{H_{i}}(\beta_{a_{n}})$, by Remark \ref{Rem1.MCI}. In either case, $H_i$ has a cycle. For instance, if $k,l \in N_{H_{i}}(\alpha_{a_{n}})$ then edges $\{\{\alpha_{a_{i+1}},k\},\{\alpha_{a_{i+1}},l\},\{\alpha_{a_{n}},k\},\{\alpha_{a_{n}},l\}\} \in E(H_i)$ form a cycle. This is a contradiction by Lemma \ref{Lemma2.an}. Thus $\{\alpha_{a_{i+1}},\beta_{a_{i+1}}\} \cap \{\alpha_{a_{n}},\beta_{a_{n}}\} \neq \emptyset$ as desired. 
\end{proof}

\begin{lemma} \label{lemma2.CUG}
    Let $G = G_1 \sqcup \cdots \sqcup G_n$ be a unicyclic graph. Then edge-binomials of $G$ form a $d$-sequence if and only if edge-binomials of $G_i$ form a $d$-sequence and for $j \neq i$, $J_{G_j}$ is a complete intersection.
\end{lemma} 
\begin{proof}
Assume that edge-binomials of $G_i$ form a $d$-sequence and for $j \neq i$, $J_{G_j}$ is a complete intersection. Note that if $J_{G}$ is a complete intersection, then edge-binomials of $G$ form a regular sequence. Thus edge-binomials associated with $G_j$'s for all $j \neq i$ form a condition independent $d$-sequence (independent of order). We take a sequence of edge-binomials of $G_j$ for all $j \neq i$ (independent of order), then edge-binomials of $G_i$ in the same as it forms a $d$-sequence. Then the above sequence forms a  $d$-sequence in $S$. Thus, $J_G$ is generated by a $d$-sequence.
    
It follows from \cite[Corollary 1.2]{EHH}  that if $G$ is not a complete intersection, then $G$ has a vertex of degree $3$ or all the vertices of $G$ have degree $2$.    Conversely, Assume that none of the edge-binomials of $G_i$ form a $d$-sequence or there exists $j$ and $k$ such that $J_{G_j}$ and $J_{G_k}$ are not complete intersection  for some $j \neq k$. It is enough to prove that edge-binomials of $J_{G_j \sqcup G_k}$ do not form a $d$-sequence. There are two cases for $G_j \sqcup G_k$
\begin{enumerate}[(i)] 
    \item  both $G_j$ and $G_j$ has at least one vertex of degree $3$,
    \item $G_j$ has a vertex of degree $3$, and all the vertices of $G_k$ have degree $2$.
\end{enumerate}
From Lemma \ref{Lemma2.T}, it follows that any sequence of edge-binomials of graph (i) or graph (ii) do not form a $d$-sequence.
\end{proof}

Recall that $\mathcal{T}_m$ denotes the class of trees on $[n]$ vertices having degree sequence $(m,2,\dots, 1)$ or $(m,1,\dots, 1)$, where $m \geq 2$, \cite{AN}. The following version of $\mathcal{T}_m$ will be useful in proving results. 

\begin{notation} \label{defPK} 
$\mathcal{T}_m$ denote the class of graphs with vertex set and edge set as below:
$$V(G) = \{k_0,p_{1,1}, \ldots,p_{1,{s(1)+1}}, p_{2,1}, \ldots,p_{2,{s(2)+1}}, \ldots, p_{m,1}, \ldots,p_{m,{s(m)+1}}\},$$
with $s(i) \geq 0$ for all $1 \leq i \leq m$, and edge set  
$$E(G) = \{\{k_0,p_{i,{1}} \mid i = 1,\ldots,m\} \cup \bigcup_{i=1}^{m} \{p_{i,{j}},p_{i,{j+1}} \mid j=1,\ldots,s(i)\}\}.$$
\end{notation} 

\begin{remark}\label{rem2.PtK1m}\cite[Theorem 2.1]{AN}
If $G \in \mathcal{T}_{m}$ then edge-binomials of $G$ form a $d$-sequence. 
\end{remark}

In the following theorem, we construct unicyclic graphs whose edge-binomials form a $d$-sequence. 

\begin{Theorem} \label{Thm2.U}
Let $H \in \mathcal{T}_{m}$ be a tree. Consider a unicyclic graph $G$ constructed by adding
\begin{enumerate}[(a)]
    \item an edge between a pendant vertex of $H$ and the center of $H$, or
    \item an edge between two pendant vertices of $H$, or
    \item an edge between the center of $H$ and an internal vertex of $H$.
\end{enumerate}
Then edge-binomials of $G$ form a $d$-sequence. In particular, $J_G$ is of linear type.
\end{Theorem}
\begin{proof}
(a). Suppose $G$ is obtained by adding an edge between a pendant vertex of $H$ and the center of $H$. Let $e_1=\{k_{0},p_{k,{s(k)+1}}\}$, for some $k$ with $s(k) \geq 1$, be an edge of $G$. Set $d_0 =  0 \in S$. Consider the following sequence of edge-binomials $d_1,\ldots,d_n$ of $G$, where the first $n-1$ elements are edge-binomials of $H$ with same order as in \cite[Theorem 2.1]{AN} and $d_n= f_{e_1}$. From Remark \ref{rem2.PtK1m}, it follows that the binomial edge ideal of  $H$ is generated by $d$-sequence. Thus, it is enough to prove that $(d_0,d_1,\ldots,d_{i}):d_{i+1}d_{n} = (d_0,d_1,\ldots,d_{i}):d_{n}$ for all $i \leq n-1$. For $i=n-1$, the equality follows from Lemma \ref{Lemma1.RD}. For $i \leq n-2$, the equality follows from Remark \ref{Rem1.MCI} and Lemma \ref{Lemma2.T}. 

(b). Suppose $G$ is obtained by adding an edge between two pendant vertices of $H \in \mathcal{T}_{m}$ with $m \geq 3$. Then, $G$ can be viewed as a graph obtained by adding an edge between a pendant vertex of $H'$ and the center of $H'$, where $H' \in \mathcal{T}_{m-1}$. If $m=2$, then G is a cycle. From \cite[Theorem 4.4]{JAR} it follows that edge-binomials of $C_n$ form a $d$-sequence.  Thus $J_G$ is generated by  $d$-sequence.

(c). Suppose $G$ is obtained by an edge between the center of $H$ and an internal vertex of $H$. Let $e_2 = \{u,v\}$ be an edge of $G$ such that $\deg_{G}(u) = m+1$ and $\deg_{G}(v) =3$. Clearly, $G \setminus \{e\} = H$. From Remark \ref{rem2.PtK1m} it follows that edge-binomials of $H$ form a $d$-sequence. Consider the following sequence of edge-binomials $d_1,\ldots,d_n$ of $G$, where the first $n-1$ elements are edge-binomials of $H$ with the same order as in \cite[Theorem 2.1]{AN} and $d_n= f_e$. The $d$-sequence conditions hold in a similar way as (a). 
\end{proof}

\begin{remark} \label{rem3.iG}
Let $H \in \mathcal{T}_{m}$ be a tree. 
\begin{enumerate}[(a)]
    \item If $G$ is a unicyclic graph as mentioned in Theorem \ref{Thm2.U}(a), then one has   
$$i(G)= \begin{cases}
    \sum_{i=1}^{m}{s_{(i)}}  & \text{if girth}(G) = 3\\
    2+\sum_{i=1}^{m}{s_{(i)}} & \text{if girth}(G) \geq 4.
\end{cases}
$$
\item If $G$ is a unicyclic graph as mentioned in Theorem \ref{Thm2.U}(c), then one has  
$$i(G)= \begin{cases}
    \sum_{i=1}^{m}{s_{(i)}}  & \text{if girth}(G) = 3\\
    1+\sum_{i=1}^{m}{s_{(i)}} & \text{if girth}(G) \geq 4.
\end{cases}
$$
\end{enumerate}
\end{remark}

Below we illustrate two types of graphs obtained from $G \in \mathcal{T}_2$ whose (parity) edge-binomials form a $d$-sequence. Let $k_0$ be the center of $G$. $G_1$ is a graph obtained by adding an edge $e_1$ between a pendant vertex of $G$ and the center of $G$. $G_2$ is a graph obtained by adding an edge $e_2$ between an internal vertex of $G$ and the center of $G$. 
\begin{figure}[ht] \label{fig1}
    \centering
\begin{tikzpicture}[x=0.75pt,y=0.75pt,yscale=-1,xscale=1]
\draw    (280.42,210.25) -- (243,210.42) ;
\draw  (318.42,236.75) -- (280.42,210.25) ;
\draw    (279.42,269.25) -- (241.42,269.25) ;
\draw   (279.42,269.25) -- (318.42,236.75) ;
\draw  (424.42,209.25) -- (387,209.42) ;
\draw  (462.42,235.75) -- (424.42,209.25) ;
\draw (423.42,268.25) -- (385.42,268.25) ;
\draw   (423.42,268.25) -- (462.42,235.75) ;
\draw   (574.42,210.25) -- (537,210.42) ;
\draw    (612.42,236.75) -- (574.42,210.25) ;
\draw    (573.42,269.25) -- (535.42,269.25) ;
\draw    (573.42,269.25) -- (612.42,236.75) ;
\draw[draw=red]    (573.42,269.25) -- (574.42,210.25) ;
\draw[draw=red]    (385.42,268.25) -- (424.42,209.25) ;

\draw (277.42,196.5) node [anchor=north west][inner sep=0.75pt]  [font=\scriptsize]  {\textcolor{blue}{$k_{0}$}};
\draw (317.42,225.5) node [anchor=north west][inner sep=0.75pt]  [font=\scriptsize]  {$p_{1,{1}}$};
\draw (279.42,269.25) node [anchor=north west][inner sep=0.75pt]  [font=\scriptsize]  {$p_{1,{2}}$};
\draw (223.42,258) node [anchor=north west][inner sep=0.75pt]  [font=\scriptsize]  {$p_{1,{3}}$};
\draw (226.42,196.5) node [anchor=north west][inner sep=0.75pt]  [font=\scriptsize]  {$p_{2,{1}}$};
\draw (421.42,195.5) node [anchor=north west][inner sep=0.75pt]  [font=\scriptsize]  {\textcolor{blue}{$k_{0}$}};
\draw (461.42,224.5) node [anchor=north west][inner sep=0.75pt]  [font=\scriptsize]  {$p_{1,{1}}$};
\draw (423.42,268.25) node [anchor=north west][inner sep=0.75pt]  [font=\scriptsize]  {$p_{1,{2}}$};
\draw (367.42,258) node [anchor=north west][inner sep=0.75pt]  [font=\scriptsize]  {\textcolor{blue}{$p_{1,{3}}$}};
\draw (370.42,196.5) node [anchor=north west][inner sep=0.75pt]  [font=\scriptsize]  {$p_{2,{1}}$};
\draw (571.42,196.5) node [anchor=north west][inner sep=0.75pt]  [font=\scriptsize]  {\textcolor{blue}{$k_{0}$}};
\draw (611.42,225.5) node [anchor=north west][inner sep=0.75pt]  [font=\scriptsize]  {$p_{1,{1}}$};
\draw (573.42,269.25) node [anchor=north west][inner sep=0.75pt]  [font=\scriptsize]  {\textcolor{blue}{$p_{1,{2}}$}};
\draw (517.42,258) node [anchor=north west][inner sep=0.75pt]  [font=\scriptsize]  {$p_{1,{3}}$};
\draw (520.42,196.5) node [anchor=north west][inner sep=0.75pt]  [font=\scriptsize]  {$p_{2,{1}}$};
 
\draw (560,290) node [anchor=north west][inner sep=0.75pt]  [font=\normalsize]  {$G_2$};

\draw (560,235) node [anchor=north west][inner sep=0.75pt]  [font=\scriptsize]  {\textcolor{blue}{$e_2$}};

\draw (410,290) node [anchor=north west][inner sep=0.75pt]  [font=\normalsize]  {$G_1$};

\draw (390,235) node [anchor=north west][inner sep=0.75pt]  [font=\scriptsize]  {\textcolor{blue}{$e_1$}};

\draw (270,290) node [anchor=north west][inner sep=0.75pt]  [font=\normalsize]  {$G$};

\filldraw[black] (280.42,210.25) circle (1.5pt) ;
\filldraw[black] (243,210.42) circle (1.5pt) ;
\filldraw[black] (318.42,236.75) circle (1.5pt) ;
\filldraw[black] ( 279.42,269.25) circle (1.5pt) ;
\filldraw[black] ( 241.42,269.25)  circle (1.5pt) ;
\filldraw[black] (424.42,209.25) circle (1.5pt) ;
\filldraw[black] ( 387,209.42) circle (1.5pt) ;
\filldraw[black] (462.42,235.75)  circle (1.5pt) ;
\filldraw[black] (423.42,268.25) circle (1.5pt) ;
\filldraw[black] (385.42,268.25) circle (1.5pt) ;
\filldraw[black] (574.42,210.25   ) circle (1.5pt) ;
\filldraw[black] (537,210.42) circle (1.5pt) ;
\filldraw[black] (612.42,236.75) circle (1.5pt) ;
\filldraw[black] (573.42,269.25) circle (1.5pt);
\filldraw[black] (535.42,269.25) circle (1.5pt);
\end{tikzpicture}
\caption{The graph $G$ with degree sequence $(2,2,2,1,1)$, $G_1=G\cup \{\textcolor{blue}{e_1}\}$, and $G_2=G\cup \{\textcolor{blue}{e_2}\}$.}
\end{figure}

\begin{remark}
The authors in \cite[Theorem 4.4.]{JAR} characterized unicyclic graphs whose binomial edge ideals are almost complete intersections. Also, the same authors in \cite{JAR20} classified almost complete intersection unicyclic graphs into two types of graphs called $G_{1}$-type and $G_{2}$-type. One can see that the $G_{1}$-type graphs are a subgraph of the collection of graphs $G$ as in Theorem \ref{Thm2.U}(a), and the $G_{2}$-type graphs are a subgraph of the collection of graphs $G$ as in Theorem \ref{Thm2.U}(c).
\end{remark}

\begin{notation}
 \label{Def:Hm} $\mathcal{H}_m$ denote the class of graphs with vertex set and edge set as below:
$$V(G) = \{k_0,k_1,p_{1,1}, \ldots,p_{1,{s(1)+1}}, p_{2,1}, \ldots,p_{2,{s(2)+1}}, \ldots, p_{{m+1},1}, \ldots,p_{{m+1},{s(m+1)+1}}\}$$
with $s(i) \geq 0$ for all $1 \leq i \leq m+1$, and edge set 
\begin{equation*}
    \begin{split}
        E(G) = &\{\{k_0,p_{i,{1}} \mid i = 1,\ldots,m-1\} \cup \{k_0,k_1\} \cup \{k_1,p_{i,{1}} \mid i = m,m+1\} \\
        &\bigcup_{i=1}^{m+1} \{p_{i,{j}},p_{i,{j+1}} \mid j=1,\ldots,s(i)\}\}.
    \end{split}
\end{equation*}
\end{notation}

\begin{lemma} \label{Pro:2}
Let $G$ be a connected unicyclic graph on $[n]$ and $H \in \mathcal{T}_{m}$ be a tree. If edge-binomials of $G$ form a $d$-sequence, then a unicyclic graph is one of the following forms:  
\begin{enumerate}[(a)]
    \item $G$ is obtained by adding an edge between a pendant vertex of $H$ and the center of $H$,
    \item $G$ is obtained by adding an edge between the center of $H$ and an internal vertex of $H$,
    \item $G=C_n$, where $ n \geq 3$.
\end{enumerate}
\end{lemma}
\begin{proof}
Let $d_1,\ldots , d_n$ be a sequence of edge-binomials of $G$ such that $d_1,\ldots , d_n$ form a $d$-sequence. Let $e$ be an edge associate to the edge-binomial $d_n$.   From Lemma \ref{Lemma2.an} it follows that $G\setminus\{e\}$ is a tree. From \cite[Theorem 0.1]{AN} it follows that   $G\setminus\{e\} \in \{\mathcal{T}_m, \mathcal{H}_m\}$. Clearly $d_1,\ldots , d_{n-1}$ form a $d$-sequence.
Now we consider possible cases for an edge $e$:

\textbf{Case 1:} If $e$ is an edge between two pendant vertices of a graph in $\mathcal{T}_m$. For $m\geq 3$, $G$ is isomorphic to a graph in (a). For $m = 2$, $G$ is isomorphic to $C_n$.

\textbf{Case 2:} If $e$ is an edge between a pendant vertex and the center of a graph in $\mathcal{T}_m$. Then $G$ is isomorphic to a graph in (a).

\textbf{Case 3:} If $e$ is an edge between a pendant vertex and an internal vertex of a graph in $\mathcal{T}_m$ such that the cycle of $G$ has a vertex of degree $m$. Then $G$ is isomorphic to a graph in (b).

\textbf{Case 4:} If $e$ is an edge between two pendant vertices of a graph in $\mathcal{H}_m$ such that all the vertices of the cycle in $G$ have degree $2$ except two vertices. Then $G$ is isomorphic to a graph in (b).

\textbf{Case 5:} If $e$ is an edge between a pendant vertex and the center of a graph in $\mathcal{H}_m$ such that all the vertices of the cycle in $G$ have degree $2$ except two vertices. Then $G$ is isomorphic to a graph in (b).

\textbf{Case 6:} From Lemma \ref{Lemma2.T} and Lemma \ref{Lemma2.an}, it follows that if $G$ has a vertex of degree at least three, which is not a vertex of the cycle of $G$, then any sequence of edge-binomials of $G$ does not satisfy the $d$-sequence condition.
\end{proof}

In the following lemma, we prove parity edge-binomials of the graphs considered in Theorem \ref{Thm2.U} form a $d$-sequence.

\begin{Theorem} \label{Rem2.Parity2}
Let $H \in \mathcal{T}_m$ be a tree. Assume that char$(k) \neq 2$. Let $G$ be a  unicyclic graph obtained by adding
\begin{enumerate}[(a)]
    \item an edge between a pendant vertex of $H$ and the center of $H$, or
    \item an edge between two pendant vertices of $H$, or
    \item an edge between the center of $H$ and an internal vertex of $H$.
\end{enumerate}
Then parity edge-binomials of $G$ forms a $d$-sequence. In particular, $\I_G$ is of linear type.
\end{Theorem}
\begin{proof}
    If $G$ has an even girth, the statement follows from Remark \ref{Rem1.phi}. If $G$ has an odd girth, then we take a sequence of parity edge-binomials in the same order as edge-binomials taken in Theorem \ref{Thm2.U}. It is enough to prove that $(d_0,d_1,\ldots,d_{i}):d_{i+1}d_{n} = (d_0,d_1,\ldots,d_{i}):d_{n}$ for all $i$, since the graph associated with parity edge-binomials $d_1,\ldots,d_{n-1}$ is a bipartite graph. For $i =n-1$, the statement follows from Lemma \ref{Lemma1.RD}. For $i < n-1$ the statement follows from Remark \ref{Rem1.phi} and Lemma \ref{Rem1.MCI}.
\end{proof}

\subsection*{Conclusion} \begin{proof}[Proof of Theorem~\ref{thmU.1}]
The first statement follows from Theorem \ref{Thm2.U} and Lemma \ref{Pro:2}. The second statement follows from Theorem \ref{Rem2.Parity2}. 
\end{proof}

\begin{remark}
Let $G$ be a unicyclic graph on $[6]$ with edge set $\{\{1,2\},\{2,3\},\{1,3\},\{1,4\},\{4,5\},\{4,6\}\}$. Then by using Macaulay2 \cite{M2} one can check that a sequence of parity edge-binomials in the following order $g_{12},g_{23},g_{13},g_{45},g_{46},g_{14}$ form a $d$-sequence. Moreover, one can observe that $G$ is not isomorphic to any graphs mentioned in Theorem \ref{Rem2.Parity2}.
\end{remark}

\begin{question}Classify all finite simple graphs such that their parity edge-binomials form a $d$-sequence.
\end{question}

\section{Regularity of (parity) binomial edge ideal of \textit{d}-sequence graphs} \label{sec.reg}

In this section, we obtain the regularity of the parity binomial edge ideal of $d$-sequence unicyclic graphs. We obtain the regularity of the product of the binomial edge ideal of a complete graph and the parity binomial edge ideal of a disjoint union of paths. The regularity of binomial edge ideals of unicyclic graphs is studied in \cite{S21}. For the sake of completeness, first, we provide the regularity of the binomial edge ideals of connected unicyclic graphs whose edge-binomials form a $d$-sequence.

\vspace{2mm}

\begin{lemma} \label{Lemma3.U1}
Let $H \in \mathcal{T}_{m}$ be a tree. Consider a unicyclic graph $G$ constructed by adding
\begin{enumerate}[(a)]
    \item an edge between a pendant vertex of $H$ and the center of $H$, then $\reg {S}/{J_{G}} = 2+ \sum_{i=1}^{m}{s_{(i)}} -1$.
    \item an edge between the center of $H$ and an internal vertex of $H$, then one has 
\begin{equation*}
 \reg {S}/{J_{G}} = 
    \begin{cases}
      1  + \sum_{i=1}^{m}{s_{(i)}} & \text{if girth}(G) = 3\\
     \sum_{i=1}^{m}{s_{(i)}} & \text{if girth}(G) \geq 4.
    \end{cases}       
\end{equation*}
\end{enumerate}
\end{lemma}
\begin{proof}
 If $G$ has a girth greater than or equal to $4$, then the statement follows from \cite[Corollary 4.10]{S21} and Remark \ref{Rem1.GUG}. If $G$ has girth equal to $3$, then $G$ is a flower-free connected graph. Thus the statement follows from Remark \ref{Rem.FF}. 
\end{proof}

In \cite{A2021reg} Kumar obtained an upper bound for parity binomial edge ideals of a non-bipartite graph $G$ such that $G \setminus e$ is a bipartite graph, where $e \in E(G)$. We compute the regularity of parity binomial edge ideals of connected unicyclic graphs whose parity edge-binomials form a $d$-sequence. By Remark \ref{Rem1.phi}, one can consider only unicyclic graphs with odd girth.

\begin{lemma} \label{lemma5.PU1}
Let $H \in \mathcal{T}_{m}$ be a tree. Let $G$ be a unicyclic graph with an odd girth. If $G$ is obtained by adding an edge between a pendant vertex of $H$ and the center of $H$. Then one has $$\reg {S}/{\mathcal{I}_{G}} = 2+ \sum_{i=1}^{m}{s_{(i)}}.$$
\end{lemma}
\begin{proof}
Let $e=\{u,v\}$ be an edge of the cycle in $G$ such that $\deg_G(u)=\deg_G(v)=2$. Consider the following short exact sequence 
\begin{equation} \label{eq:PT1}
    0 \longrightarrow \frac{S}{(\mathcal{I}_{G \setminus e}):g_{e}}(-2) \longrightarrow \frac{S}{\mathcal{I}_{G \setminus e}}
  \longrightarrow \frac{S}{\mathcal{I}_{G}}   \longrightarrow  0.
\end{equation}
Clearly, $G\setminus e$ is a tree. From Remarks \ref{Rem1.MCI}(c) and \ref{rem1.Reg2} it follows that,
$$\reg {S}/{((\mathcal{I}_{G \setminus e}):g_{e})} = \reg {S}/{{J}_{(G \setminus e)_e}} = \reg {S}/{J_{G \setminus e}}  = 2+ \sum_{i=1}^{m}{s_{(i)}} -1.$$
Thus applying Lemma \ref{Lemma1.Reg}(c) to short exact sequence (\ref{eq:PT1}) gives  $\reg {S}/{\mathcal{I}_{G}} = 2+ \sum_{i=1}^{m}{s_{(i)}}.$  
\end{proof}

\begin{lemma} \label{lemma5.PU2}
Let $H \in \mathcal{T}_{m}$ be a tree. Let $G$ be a unicyclic graph with an odd girth. If $G$ is obtained by adding an edge between the center of $H$ and an internal vertex of $H$, then 
\begin{equation*}
 \reg {S}/{\mathcal{I}_{G}} = 
      2  + \sum_{i=1}^{m}{s_{(i)}}-1.  
\end{equation*}
\end{lemma}
\begin{proof}
Let $G$ have the girth greater than or equal to $4$ and $e=\{u,v\}$ be an edge of the cycle in $G$ such that $\deg_G(u)=\deg_G(v)=2$. 
Then from Remark \ref{Rem1.MCI}(c) and {\cite[Lemma 3.1]{AN}} it follows that 
$$\reg {S}/{((\mathcal{I}_{G \setminus e}):g_{e})} = \reg {S}/{{J}_{(G \setminus e)_e}} = \reg {S}/{J_{G \setminus e}}  = \sum_{i=1}^{m}{s_{(i)}}.$$
Therefore, applying Lemma \ref{Lemma1.Reg}(c) to exact sequence (\ref{eq:PT1}) yields  $\reg {S}/{\mathcal{I}_{G}} = 2 + \sum_{i=1}^{m}{s_{(i)}} -1.$

If $G$ have girth equal to $3$ then choose an edge of the cycle in $G$, $e=\{u,v\}$ such that $\deg_G(u)= m+1$ and $\deg_G(v)=2$. $G\setminus e$ is a graph in $\mathcal{T}_{m+1}$, thus from Remark \ref{Rem1.phi} and Remark \ref{rem1.Reg2} it follows that 
$$\reg {S}/{\mathcal{I}_{G \setminus e}} = \reg {S}/{J_{G \setminus e}} =1+ \sum_{i=1}^{m}{s_{(i)}}.$$
By  Remark \ref{Rem1.MCI}(c) we get $J_{(G\setminus e)_e}$ is a graph obtained by gluing $K_{m}$ and gluing paths at free vertices. By Remark \ref{Rem1.GUG} one obtains 
$$\reg {S}/{((\mathcal{I}_{G \setminus e}):g_{e})} = \sum_{i=1}^{m}{s_{(i)}}.$$
Therefore, applying Lemma \ref{Lemma1.Reg}(c) to exact sequence (\ref{eq:PT1}) yields  $\reg {S}/{\mathcal{I}_{G}} = 1 + \sum_{i=1}^{m}{s_{(i)}}.$
Hence we conclude the proof.
\end{proof}

Next, we obtain the regularity of the product of parity binomial edge of disjoint union of paths and binomial edge ideal a complete graph.

\begin{Theorem} \label{Thm3.parityIJ}
Let $K_{m}$ be a complete graph and $H=\{P_1',\ldots P_t'\}$ be a disjoint union of paths such that:
\begin{enumerate}
    \item for any $i$, if $K_m \cap P_i' \neq \emptyset$, then $V(K_m) \cap V(P_i') = v_i$, for some $v_i$ which is free vertex in $P_i'$;
    \item $V(K_m) \cap V(P_i') \cap V(P_j')= \emptyset$, for all distinct $i$ and $j$.
\end{enumerate}
Let $n$ be the number of edges in $H$. Then, for any $n \geq 1$ and for any $m \geq 2$, we have 
$$\reg \frac{S}{\I_{H}J_{K_{m}}} = 2+n.$$
\end{Theorem}
\begin{proof}
Since $H$ is a bipartite graph, one can write $V(H) = L_1 \sqcup L_2$ such that $L_2 \cap V(K_m) = \emptyset$. Take $V_1 = L_1 \cup V(K_m)$ and $V_2 =L_2$ as partition of $V(H)\cup V(K_m)$. From Remark \ref{Rem1.phi} it follows that $\phi(\I_HJ_{K_m})=J_HJ_{K_m}$. As desired, the statement follows from \cite[Theorem 3.1]{AN}.   
\end{proof}

\section{Regularity of powers of binomial edge ideals} \label{sec.BEI}

In this section, we obtain precise expressions for the regularity of powers of the binomial edge ideal of $d$-sequence unicyclic graphs. The regularity of powers of the binomial edge ideals of cycle graphs is computed in \cite{JAR20}. We recall the statement below for the sake of completeness.

\begin{remark}
\cite[Theorem 3.6]{JAR20} Let $G = C_n$ be a cycle. Then, for all $s\geq 1$, $\reg(S/J_{G}^s ) = 2s + n - 4$.
\end{remark}

Next, we compute the regularity of powers of binomial edge ideals of connected unicyclic graphs whose edge-binomial forms a $d$-sequence. 

\begin{Theorem} \label{Thm4.UC1}
Let $H \in \mathcal{T}_{m}$ be a tree. Let $G$ be a unicyclic graph obtained by adding an edge between a pendant vertex of $H$ and the center of $H$. Let $f_1,\ldots,f_n$ be a $d-$sequence edge-binomials of $G$ with $f_0 =  0 \in S$. Then, for any $i=0,1,\ldots,n-1$, we have
$$\reg \frac{S}{(f_1,\ldots,f_i)+J_{G}^s} = 2s + \sum_{j=1}^{m}{s_{(j)}}-1, \text{ for all } s\geq 1.$$
In particular, $\reg {S}/{J_{G}^s} = 2s + \sum_{j=1}^{m}{s_{(j)}}-1$, for all $s\geq 1$.
\end{Theorem}
\begin{proof} 
We consider a sequence $f_1,\ldots,f_n$ of edge-binomials of $G$ as the same sequence as in the Theorem \ref{Thm2.U}(a).
The proof is by induction on $s$. For $s=1$, the assertion follows from  Lemma \ref{Lemma3.U1}(a). We can assume that the assertion holds for $s-1$. The statement for $s$ is proved by descending induction on $i$. For $i=n-1$, the statement is verified independently in Lemma \ref{Lemma4.UC1dn}. Assume that the assertion holds for $i+1$. To prove the statement for $s$ and $i$, consider the following short exact sequence
\begin{equation} \label{eq:U1}
\begin{split}
     0 \longrightarrow \frac{S}{(f_1,\ldots,f_{i})+J_{G}^{s}:f_{i+1}}(-2) & \longrightarrow \frac{S}{(f_1,\ldots,f_{i})+J_{G}^{s}} \\
      & \longrightarrow \frac{S}{(f_1,\ldots,f_{i+1})+J_{G}^{s}}   \longrightarrow  0. 
\end{split}
\end{equation}
By the induction hypothesis on $i$, it follows that $\reg {S}/{(f_1,\ldots,f_{i+1})+J_{G}^{s}} = 2s + \sum_{j=1}^m{s_{(j)}}-1$. From Lemma \ref{Lemma1.PD} it follows that $((f_1,\ldots,f_{i})+J_{G}^{s}):f_{i+1}=((f_1,\ldots,f_{i}):f_{i+1})+J_{G}^{s-1}$. From Remark \ref{Rem1.MCI}(a) one has $((f_1,\ldots,f_{i}):f_{i+1})+J_{G}^{s-1}= I+J$, where $I = (f_1,\ldots,f_{i})+J_{G}^{s-1}$ and $J = J_{K_n}$ for some $n<m$ ($J_{K_n} = ( f_{kl} \mid k,l \in N_{H_i}(\alpha(f_{i+1})) \text{ or } k,l \in N_{H_i}(\beta(f_{i+1}))$).
Consider the following two cases:

\textbf{Case 1:}
If  $((f_1,\ldots,f_{i}):f_{i+1})+J_{G}^{s-1} = I$. From the induction hypothesis on $s$, it follows that $\reg{S}/{I} = 2(s-1) + \sum_{j=1}^{m+1}{s_{(j)}}-1$. Applying Lemma \ref{Lemma1.Reg}(b) to the short exact sequence (\ref{eq:U1}) yields 
 $$\reg \frac{S}{(f_1,\ldots,f_{i})+J_{G}^{s}}= 2s + \sum_{j=1}^{m+1}{s_{(j)}}-1.$$
 
\textbf{Case 2:}
If $((f_1,\ldots,f_{i}):f_{i+1})+J_{G}^{s-1} =I + J$. Consider the short exact sequence 
  \begin{equation} \label{eq:IJ}
    0 \longrightarrow \frac{S}{I \cap J} \longrightarrow \frac{S}{I} \oplus \frac{S}{J}
  \longrightarrow \frac{S}{I+J}   \longrightarrow  0.
 \end{equation}
Since $J$ is a complete graph, from Remark \ref{Rem.FF} it follows that $\reg {S}/{J} = 1$. We claim that $I \cap J = J \cdot( x_{k_{0}},y_{k_{0}},J_{\mathcal{P}} )$, where $\mathcal{P}$ is the induced subgraph of $G$ with vertex set $G\setminus \{k_0\}$ and $k_0$ denotes the center of $H$.

The proof of the claim is similar to the proof of Claim 4.1 of \cite[Theorem 4.1]{AN}. From Theorem \cite[Theorem 3.1]{AN} it follows that $\reg S/(I\cap J) = 2+ \sum_{j=1}^{m}{s_{(j)}}$. Then applying Lemma \ref{Lemma1.Reg}(c) to the short exact sequence (\ref{eq:IJ}) yields $\reg{{S}/{(I+J)}} = 2(s-1)+ \sum_{j=1}^{m}{s_{(j)}}-1$, for all $s\geq 2$. Again applying  Lemma \ref{Lemma1.Reg}(b) to exact sequence (\ref{eq:U1}) yields  $$\reg{\frac{S}{(f_1,\ldots,f_{i})+J_{G}^{s}}} = 2s + \sum_{j=1}^{m}{s_{(j)}}-1, \text{ for all } s \geq 2.$$ 
\end{proof}

\begin{lemma} \label{Lemma4.UC1dn}
With the hypothesis as in Theorem \ref{Thm4.UC1}, and for all $s>1$, we have
$$\reg \frac{S}{(f_1,\ldots,f_{n-1})+ f_{n}^{s}} = 2s + \sum_{j=1}^{m}{s_{(j)}}-1.$$
\end{lemma}
\begin{proof}
From Lemma \ref{Lemma3.U1}(a), it follows that $\reg {S}/{J_{G}} = 2 + \sum_{j=1}^{m}s_{(j)}-1$. By using Theorem \ref{Thm2.U}(a) and Remark \ref{rem1.Reg2}, we obtain that $\reg {S}/{(f_1,\ldots,f_{n-1})} = 2 + \sum_{j=1}^{m}s_{(j)}$. Therefore, applying Lemma \ref{Lemma1.Reg}(a) to the short exact sequence 
\begin{equation} \label{eq:E1}
    0 \longrightarrow \frac{S}{(f_1,\ldots,f_{n-1}):f_n}(-2) \longrightarrow \frac{S}{(f_1,\ldots,f_{n-1})} 
  \longrightarrow \frac{S}{(f_1,\ldots,f_{n})}   \longrightarrow  0,
\end{equation}
yields  $$\reg \frac{S}{(f_1,\ldots,f_{n-1}):f_{n}} = \sum_{j=1}^{m}s_{(j)}.$$ 
Since $f_1,\ldots,f_{n}$ form a $d$-sequence, one has $(f_1,\ldots,f_{n-1}): f_{n}^{s} = (f_1,\ldots,f_{n-1}): f_{n}$. Now, look at the following exact sequence 
\begin{equation} \label{eq:E4}
\begin{split}
    0 \longrightarrow \frac{S}{(f_1,\ldots,f_{n-1}):f_{n}^{s}}(-2s) & \longrightarrow \frac{S}{(f_1,\ldots,f_{n-1})} \\
 & \longrightarrow \frac{S}{(f_1,\ldots,f_{n-1})+f_{n}^{s}}   \longrightarrow  0. 
\end{split}
\end{equation}
Again, applying Lemma \ref{Lemma1.Reg}(c) to the short exact sequence (\ref{eq:E4}) one can obtain the desired result.
\end{proof}

\begin{remark}
If $m=2$ in Theorem \ref{Thm4.UC1} then $G$ is a balloon graph.  Therefore, we obtained \cite[Remark 3.15]{JAR20} that $\reg S/J_G^s = 2s + n-4$, for all $s \geq 1$.   
\end{remark}

\begin{Theorem} \label{Thm4.U2}
Let $H \in \mathcal{T}_{m}$ be a tree. Let $G$ be a unicyclic graph obtained by adding an edge between the center of $H$ and an internal vertex of $H$. Let $f_1,\ldots f_n$ be a $d$-sequence edge-binomials of $G$ with $f_0 =  0 \in S$. Then, for any $i=0,1,\ldots,n-1$, we have
$$\reg \frac{S}{(f_1,\ldots,f_{i})+ J_{G}^{s}} = 2s + \sum_{j=1}^{m}{s_{(j)}}-1,  \text{ for all } s\geq 2.$$
In particular, $\reg {S}/{J_{G}^s} = 2s+ \sum_{j=1}^{m}{s_{(j)}}-1$, for all $s\geq 2$.
\end{Theorem}
Take $f_1,\ldots,f_n$ to be a sequence of edge-binomials in the same order as in Theorem \ref{Thm2.U}(c). To complete the proof of Theorem \ref{Thm4.U2}, we need to prove the following lemmas.

\begin{lemma} \label{Lemma4.U2s=2}
With the hypothesis as in Theorem \ref{Thm4.U2}, and for any $i=0,1,\ldots,n-1$, we have
$$\reg \frac{S}{(f_1,\ldots,f_i)+J_{G}^2} = 4+ \sum_{j=1}^{m}{s_{(j)}}-1.$$
\end{lemma}
\begin{proof}
From Lemma \ref{Lemma3.U1}(b) one can obtain the regularity of the module $S/(f_1,\ldots,f_{n})$. By Theorem \ref{Thm2.U}(c) the binomial edge ideal $( f_1,\ldots,f_{n-1}) $ is associated to a graph in  $\mathcal{T}_{m}$. Thus, by Remark \ref{rem1.Reg2} it follows that $\reg S/(f_1,\ldots,f_{n-1}) = 2+ \sum_{j=1}^{m}{s_{(j)}}$. Therefore, by applying Lemma \ref{Lemma1.Reg}(a) to the short exact sequence (\ref{eq:E1}) yields
$$\reg \frac{S}{(f_1,\ldots,f_{n-1}):f_n} = \sum_{j=1}^{m}{s_{(j)}}. $$
Now, the proof is by descending induction on $i$.
For $i=n-1$, consider the short exact sequence (\ref{eq:E4}) with $s=2$. Since the sequence $f_1,\ldots,f_{n}$ is a $d$-sequence, one has 
 $$\reg \frac{S}{(f_1,\ldots,f_{n-1}):f_{n}^{2}}(-4) = 4+ \sum_{j=1}^{m}{s_{(j)}}-1 > \reg \frac{S}{(f_1,\ldots,f_{n-1})}.$$
Applying Lemma \ref{Lemma1.Reg}(c) to the short exact sequence (\ref{eq:E4}) for $s=2$ yields that 
\[
\reg \frac{S}{(f_1,\ldots,f_{n-1})+J_{G}^{2}} = \reg \frac{S}{(f_1,\ldots,f_{n-1})+f_{n}^{2}} =4+ \sum_{j=1}^{m}{s_{(j)}}-1.
\]
Next, assume that the assertion holds for $i+1$. Consider the short exact sequence (\ref{eq:U1}) for $s=2$. From Lemma \ref{Lemma1.PD}, it follows that $((f_1,\ldots,f_{i})+J_{G}^{2}):f_{i+1} = ((f_1,\ldots,f_{i}):f_{i+1})+J_{G}$. Following the same approach as in proof of Theorem \ref{Thm4.U2} (Case $2$), we get $\reg {S}/{((f_1,\ldots,f_{i})+J_{G}^{2}):f_{i+1}} = 4+ \sum_{j=1}^{m}{s_{(j)}}-1$. By the induction hypothesis on $i$ that $\reg {S}/{(f_1,\ldots,f_{i+1})+J_{G}^{2}}= 4+ \sum_{j=1}^{m}{s_{(j)}}-1$. Applying Lemma \ref{Lemma1.Reg}(b) to exact sequence (\ref{eq:U1}) for $s=2$, yields 
\[
\frac{S}{(f_1,\ldots,f_{i})+J_{G}^{2}}=4+ \sum_{j=1}^{m}{s_{(j)}}-1,
\]
as desired.
\end{proof}

\begin{lemma} \label{Lemma4.U2p2}
With the hypothesis as in Theorem \ref{Thm4.U2}, and for all $s>2$, we have 
$$\reg \frac{S}{(f_1,\ldots,f_{n-1})+ f_{n}^{s}} = 2s + \sum_{j=1}^{m}{s_{(j)}}-1.$$
\end{lemma}
\begin{proof}
Since $f_1,\ldots,f_{n}$ form a $d$-sequence, one has $(f_1,\ldots,f_{n-1}): f_{n}^{s} = (f_1,\ldots,f_{n-1}): f_{n}$. From Lemma \ref{Lemma4.U2s=2}, we know the regularity of modules ${S}/{(f_1,\ldots,f_{n-1}):f_n}$ and ${S}/{(f_1,\ldots,f_{n-1})}$. Then applying Lemma \ref{Lemma1.Reg}(c) to the short exact sequence (\ref{eq:E4}) one yields   $\reg {S}/{((f_1,\ldots,f_{n-1})+f_{n}^{s})} = 4+ \sum_{j=1}^{m}{s_{(j)}}-1$.
\end{proof}

\begin{proof}[Proof of Theorem~\ref{Thm4.U2}] 
We use induction on $s$ to prove the statement. For $s=2$, the statement follows from  Lemma \ref{Lemma4.U2s=2}. Assume that the assertion holds for $s-1$. The statement for $s$ is proved by descending induction on $i$. For $i=n-1$, the statement is verified separately in Lemma \ref{Lemma4.U2p2}. Next, assume that the assertion holds for $i+1$. Consider the short exact sequence (\ref{eq:U1}).
By the induction hypothesis on $i$ it follows that ${S}/{((f_1,\ldots,f_{i+1})+J_{G}^{s})} = 2s +\sum_{j=1}^{m}{s_{(j)}}-1$. By the similar argument as in Theorem \ref{Thm4.UC1} we get $((f_1,\ldots,f_{i}):f_{i+1})+J_{G}^{s-1} =(f_1,\ldots,f_{i})+J_{G}^{s-1}+J_{K_{n}}$, for some $n < m$. Set $I= (f_1,\ldots,f_{i})+J_{G}^{s-1}$ and $J=J_{K_{n}}$. Consider the following two cases: 

\textbf{Case 1:}
If  $((f_1,\ldots,f_{i}):f_{i+1})+J_{G}^{s-1} = I$, then from induction hypothesis it follows that $\reg {S}/{I} = 2s +\sum_{j=1}^{m}{s_{(j)}}-1$. Now, from Lemma \ref{Lemma1.Reg}(b) it follows that $$\reg \frac{S}{(f_1,\ldots,f_{i})+J_{G}^{s}}= 2s +\sum_{j=1}^{m}{s_{(j)}}-1.$$ 

\textbf{Case 2:}
If $((f_1,\ldots,f_{i}):f_{i+1})+J_{G}^{s-1} =I + J$. Consider the  short exact sequence (\ref{eq:IJ}). The regularity of $I$ is known from the induction hypothesis. From Remark \ref{Rem.FF} it follows that $\reg {S}/{J} = 1$, since $J$ is a complete graph. We claim that $I \cap J = J \cdot( x_{k_{0}},y_{k_{0}},J_{\mathcal{P}} )$, where $\mathcal{P}$ denotes the induced subgraph of $G$ with vertex set $G\setminus\{k_0\}$ and $k_0$ is the center of $G$. 

The proof of the claim is similar to the proof of Claim 4.1 of \cite[Theorem 4.1]{AN}.
It follows from Theorem \cite[Theorem 3.1]{AN} that $\reg S/(I \cap J) = 2 + \sum_{j=1}^{m}{s_{(j)}}$. By applying Lemma \ref{Lemma1.Reg}(c) to the short exact sequence (\ref{eq:IJ}) we  obtain that $\reg {S}/({I+J}) = 2(s-1) + \sum_{j=1}^{m}{s_{(j)}}-1$, for all $s\geq 3$, i.e. 
$$\reg{\frac{S}{(f_1,\ldots,f_{i})+J_{G}^{s}: f_{i+1}}} = 2(s-1) + \sum_{j=1}^{m}{s_{(j)}}-1,$$
for all $s\geq 3$. Thus by applying Lemma \ref{Lemma1.Reg}(b) to the short exact sequence (\ref{eq:U1}) one can yields  $\reg {S}/{((f_1,\ldots,f_{i})+J_{G}^{s})} = 2s + \sum_{j=1}^{m}{s_{(j)}}-1$, for $s \geq 3$. 
\end{proof}

\begin{remark}
Suppose $G$ in Theorems \ref{Thm4.UC1} and \ref{Thm4.U2} is isomorphic to a connected graph on $[n]$, which is obtained by adding an edge between two vertices of a path $P_n$. If the girth of $G$ is at least $4$, then we obtain \cite[Theorem 4.4.]{SZ} that $\reg S/J_{G}^{s} = 2s + n - 4$ for all $s \geq 2$. 
\end{remark}

\section{Regularity of powers of parity binomial edge ideals} \label{sec.PBEI}

In this section, we obtain the exact bound for the regularity of powers of the parity binomial edge ideal of $d$-sequence unicyclic graphs. As mentioned earlier, we consider only unicyclic graphs with odd girth. Kumar characterized graphs whose parity binomial edge ideals are complete intersections; namely, $\I_G$ is a complete intersection if and only if G is a path or odd cycle (cf. \cite{A2021reg}). First, we state the regularity of powers of parity binomial edge ideals of cycles. 

\begin{remark}
    Let $G$ be an odd cycle on $[n]$. Then one has $\reg S/\I_{G}^s = 2s + n-2$, for all $s\geq 1$.
\end{remark}
\begin{proof}
    The statement follows from \cite[Lemma 4.4]{BHT15}.
\end{proof}

Next, we compute the regularity of powers of parity binomial edge ideals of connected unicyclic graphs whose parity edge-binomials form a $d$-sequence.

\begin{Theorem} \label{Thm5.PPUC1}
Let $H \in \mathcal{T}_{m}$ be a tree. Let $G$ be a unicyclic graph obtained by adding an edge between the pendant vertex of $H$ and the center of $H$ such that the girth of $G$ is odd. Let $g_1,\ldots,g_n$ be a $d$-sequence parity edge-binomials of $G$ with $g_0 =  0 \in S$. Then, for any $i=0,1,\ldots,n-1$, for all $s\geq 1$, we have
$$\reg \frac{S}{(g_1,\ldots,g_i)+\mathcal{I}_{G}^s} = 2s + \sum_{j=1}^{m}{s_{(j)}}.$$
In particular, $\reg {S}/{\mathcal{I}_{G}^s} = 2s + \sum_{j=1}^{m}{s_{(j)}}$, for all $s\geq 1$.
\end{Theorem}
\begin{proof} 
We take parity edge-binomials of $G$ in the same order as edge-binomials of $G$ taken in Theorem \ref{Thm2.U}(a). The proof is by induction on $s$. The statement holds for $s=1$, by Lemma \ref{lemma5.PU1}. Assume that the statement holds for $s-1$. We prove the statement for $s$ by descending induction on $i$. For $i=n-1$, the statement holds, by Lemma \ref{Lemma5.UC1dn}. Assume that the assertion holds for $s$ and $i+1$. We need to prove the statement for $s$ and $i$. Consider the following short exact sequence \begin{equation} \label{eq:pU1}
\begin{split}
     0 \longrightarrow \frac{S}{(g_1,\ldots,g_{i})+\I_{G}^{s}:g_{i+1}}(-2) & \longrightarrow \frac{S}{(g_1,\ldots,g_{i})+\I_{G}^{s}} \\
      & \longrightarrow \frac{S}{(g_1,\ldots,g_{i+1})+\I_{G}^{s}}   \longrightarrow  0. 
\end{split}
\end{equation}

By the induction hypothesis on $i$, it follows that $\reg {S}/{(g_1,\ldots,g_{i+1})+\I_{G}^{s}} = 2s + \sum_{j=1}^m{s_{(j)}}$. From Lemma \ref{Lemma1.PD} and Remark \ref{Rem1.pcolon} it follows that  $((g_1,\ldots,g_{i}):g_{i+1})+\I_{G}^{s-1} =(g_1,\ldots,g_{i})+\I_{G}^{s-1}+J_{K_{n}}$, for some $n < m$ (where $J_{K_n} = ( f_{kl} \mid k,l \in N_{H_i}(\alpha(g_{i+1})) \text{ or } k,l \in N_{H_i}(\beta(g_{i+1}))$). 
Set $I= (g_1,\ldots,g_{i})+\I_{G}^{s-1}$ and $J=J_{K_{n}}$. Consider the following two cases:

\textbf{Case 1:}
If  $((g_1,\ldots,g_{i}):g_{i+1})+\I_{G}^{s-1} = I$. From the induction hypothesis on $s$, it follows that $\reg{S}/{I} = 2(s-1) + \sum_{j=1}^{m+1}{s_{(j)}}$. Applying Lemma \ref{Lemma1.Reg}(b) to the short exact sequence (\ref{eq:U1}) yields 
 $$\reg \frac{S}{(g_1,\ldots,g_{i})+J_{G}^{s}}= 2s + \sum_{j=1}^{m+1}{s_{(j)}}.$$

\textbf{Case 2:}
For $((g_1,\ldots,g_{i}):g_{i+1})+\I_{G}^{s-1} =I + J$.  Consider the short exact sequence (\ref{eq:IJ}). Using the induction hypothesis on $s$, regularity of $I$ follows. The regularity of $J$ follows from Remark  \ref{Rem.FF}. To find the regularity of $I \cap J$ we need to prove the following claim.

\begin{claim} \label{Claim1}
$I \cap J = J \cdot( x_{k_{0}},y_{k_{0}},\I_{\mathcal{P}} )$, where $\mathcal{P}$ denotes the induced subgraph of $G$ on vertex set $V(G)\setminus k_0$ and $k_0$ denote the center of $H$. The proof of the claim is similar to the proof of a claim $4.1$ in \cite[Theorem 4.1]{AN}. In the proof instead of edge-binomials $f_{lk_0}$ and $f_{kk_0}$ we take parity edge-binomials $g_{lk_0}$ and $g_{kk_0}$, rest of the proof is similar.
\end{claim}

Now, from Theorem \ref{Thm3.parityIJ} it follows that $\reg S/(I\cap J) = 2+ \sum_{j=1}^{m}{s_{(j)}}$. Then applying Lemma \ref{Lemma1.Reg}(c) to the short exact sequence (\ref{eq:IJ}) yields $\reg{{S}/{(I+J)}} = 2(s-1)+ \sum_{j=1}^{m}{s_{(j)}}$, for all $s > 2$. Again applying  Lemma \ref{Lemma1.Reg}(b) to exact sequence (\ref{eq:U1}) one yields  $$\reg{\frac{S}{(g_1,\ldots,g_{i})+I_{G}^{s}}} = 2s + \sum_{j=1}^{m}{s_{(j)}}, \text{ for all } s > 2.$$

For $s=2$, first we compute the regularity of  ${S}/{(\I_{G} + J_{K_n})}$ for some $n<m$ (where $J_{K_n}=( f_{kl} \mid k,l \in N_{L_1}(\alpha(g_{i+1})) \text{ or } k,l \in N_{L_1}(\beta(g_{i+1}))$). Assume that $E(G)=E(\mathcal{T}_m) \cup \{k_0,p_{1_{s(1)+1}}\}$ be an edge set of $G$. Observe that induced subgraph on the vertex set $V(G)\setminus \{p_{1_1},p_{1_2}\}$ is a tree. We partition vertex set of $V(G)\setminus \{p_{1_1},p_{1_2}\}$ into $V_1 \sqcup V_2$. Without loss of generality one assume that the vertex $k_0 \in V_1$, then we add vertices $p_{1_1}$ and $p_{1_2}$ to $V_2$. Let $e = \{p_{1_1},p_{1_2}\}$ be an edge of $G$, note that $\phi(\I_{G}+J_{K_n})=J_{G\setminus e} + g_e$. Consider the following short exact sequence
\begin{equation} \label{eq:S2}
    0 \longrightarrow \frac{S}{J_{G\setminus e}:g_e}(-2) \longrightarrow \frac{S}{J_{G\setminus e}} 
  \longrightarrow \frac{S}{(J_{G\setminus e} + g_e)}   \longrightarrow  0,
\end{equation}
$G\setminus e$ is a tree that belongs to a class of $\mathcal{T}_m$ graph. From Remark \ref{rem1.Reg2} it follows that $\reg S/{J_{G\setminus e}} = 2 + \sum_{j=1}^{m}{s_{(j)}}$. From a part of the proof of \cite[Lemma 3.3]{A2021LSS} one has $J_{G\setminus e}:g_e = J_{(G\setminus e)_e}$ and here $J_{(G\setminus e)_e}= J_{G\setminus e}$. Therefore, applying Lemma \ref{Lemma1.Reg}(c) to (\ref{eq:S2}) one yields $$\reg \frac{S}{(J_{G\setminus e} + g_e)} =  2 + \sum_{j=1}^{m}{s_{(j)}} -1.$$
From Lemma \ref{Lemma1.PD} it follows that $(g_1,\ldots,g_{i})+\I_{G}^{2}:g_{i+1} = (g_1,\ldots,g_{i}):g_{i+1} + \I_{G} = {\I_{G} + J_{K_n}}.$ Again applying Lemma \ref{Lemma1.Reg}(b) to  exact sequence (\ref{eq:U1}) for $s=2$ one yields 
$$\reg{\frac{S}{(g_1,\ldots,g_{i})+\I_{G}^{2}}} = 4 + \sum_{j=1}^{m}{s_{(j)}}.$$ 
\end{proof}

\begin{lemma} \label{Lemma5.UC1dn}
Considering the hypothesis of Theorem \ref{Thm5.PPUC1}, and for all $s>1$, we have
$$\reg \frac{S}{(g_1,\ldots,g_{n-1})+ g_{n}^{s}} = 2s + \sum_{j=1}^{m}{s_{(j)}}.$$
\end{lemma}
\begin{proof}
Let $e$ be a corresponds to a parity edge-binomial $g_n$. From Lemma \ref{Rem1.MCI}(c) it follows that $\reg S/(\mathcal{I}_{G \setminus e}:g_n) = \reg S/J_{(G \setminus e)_e} = 2 + \sum_{j=1}^{m}{s_{(j)}}-1$. By using Remark \ref{Rem1.phi} and Remark \ref{rem1.Reg2}, one can obtain that $\reg {S}/{(g_1,\ldots,g_{n-1})} = 2 + \sum_{j=1}^{m}s_{(j)}$. Applying Lemma \ref{Lemma1.Reg}(c) to the short exact sequence (\ref{eq:E4}) one can obtain the desired result.
\end{proof}

\begin{remark}
    Note that the balloon graphs are a particular case of graphs discussed in Theorem \ref{Thm5.PPUC1}. Hence we obtained the result \cite[Theorem 5.7]{SZ} as a particular case.
\end{remark}

\begin{Theorem} \label{Thm5.PPU2}
Let $H \in \mathcal{T}_{m}$ be a tree. Let $G$ be a unicyclic graph on $[n]$ with odd girth, which is obtained by adding an edge between the center of $H$ and an internal vertex of $H$. Let $g_1,\ldots g_n$ be a $d$-sequence of parity edge-binomials of $G$ with $g_0 =  0 \in S$. Then, for any $i=0,1,\ldots,n-1$, and for all $s\geq 1$, we have
$$\reg \frac{S}{(g_1,\ldots,g_{i})+ \mathcal{I}_{G}^{s}} = 2s+ \sum_{j=1}^{m}{s_{(j)}}-1.$$
In particular, $\reg {S}/{\mathcal{I}_{G}^s} = 2s+ \sum_{j=1}^{m}{s_{(j)}}-1$, for all $s\geq 1$.
\end{Theorem}
\begin{proof}
We consider parity edge-binomials of $G$ in the same order as edge-binomials of $G$ in Theorem \ref{Thm2.U}(c). Proof is by induction on $s$. For $s=1$ the statement is proved in Lemma \ref{lemma5.PU2}. Assume that the statement is true for $s-1$. The statement for $s$ is proved by descending induction on $i$. For $i=n-1$, the statement holds, by Lemma \ref{Lemma5.PU2p2}. Assume that the assertion holds for $i+1$. It remains to show that the statement holds for $s$ and $i$. Consider the short exact sequence (\ref{eq:pU1}). The statement holds for ${S}/{((g_1,\ldots,g_{i+1})+\I_{G}^{s})}$ by the reverse induction on $i$. By a similar argument as in Theorem \ref{Thm5.PPUC1} that we take $I= (g_1,\ldots,g_{i})+\I_{G}^{s-1}$ and $J=J_{K_{n}}$. Consider the following two cases:

\textbf{Case 1:}
If  $((g_1,\ldots,g_{i}):g_{i+1})+\I_{G}^{s-1} = I$, then from induction hypothesis on $s$ it follows that $\reg {S}/{I} = 2(s-1)+ \sum_{j=1}^{m}{s_{(j)}}-1$. Now, from Lemma \ref{Lemma1.Reg}(b) it follows that $$\reg \frac{S}{(g_1,\ldots,g_{i})+\I_{G}^{s}}= 2s+ \sum_{j=1}^{m}{s_{(j)}}-1.$$ 

\textbf{Case 2:}
If $((g_1,\ldots,g_{i}):g_{i+1})+\I_{G}^{s-1} =I + J$. Consider the  short exact sequence (\ref{eq:IJ}). To compute the regularity of $I \cap J$, first, we claim that $I \cap J = J \cdot( x_{k_{0}},y_{k_{0}},\I_{\mathcal{P}} )$, where $\mathcal{P}$ denotes the induced subgraph of $G$ on vertex set $V(G)\setminus k_0$ and $k_0$ denote the center of $H$.

The proof of the claim is similar to the proof of Claim 5.1 of Theorem \ref{Thm5.PPUC1}. It follows from Theorem \ref{Thm3.parityIJ} that $\reg S/(I \cap J) = 2 + \sum_{j=1}^{m}{s_{(j)}}$. By applying Lemma \ref{Lemma1.Reg}(c) to the short exact sequence (\ref{eq:IJ}) we  obtain that $\reg {S}/({I+J}) = 2(s-1) +\sum_{j=1}^{m}{s_{(j)}}-1$, for all $s\geq 2$, i.e. 
$$\reg{\frac{S}{(g_1,\ldots,g_{i})+J_{G}^{s}: g_{i+1}}} = 2(s-1) +\sum_{j=1}^{m}{s_{(j)}}-1,$$
for all $s\geq 2$. Thus by applying Lemma \ref{Lemma1.Reg}(b) to the short exact sequence (\ref{eq:pU1}) one yields 

$$\reg \frac{S}{((g_1,\ldots,g_{i})+J_{G}^{s})} = 2s + +\sum_{j=1}^{m}{s_{(j)}}-1, \text{ for all } s \geq 2.$$ 
\end{proof}

\begin{lemma} \label{Lemma5.PU2p2}
Under the assumption in Theorem \ref{Thm5.PPU2}, and for all $s>2$, we have 
$$\reg \frac{S}{(g_1,\ldots,g_{n-1})+ g_{n}^{s}} = 2s +\sum_{j=1}^{m}{s_{(j)}}-1.$$
\end{lemma}
\begin{proof}
Consider the short exact sequence (\ref{eq:E1}). From Lemma \ref{lemma5.PU2} one has the regularity of the module $S/\I_G$. The graph associated to parity edge-binomials $g_1,\ldots,g_{n-1}$ is a tree. From Remark \ref{Rem.FF} it follows that $\reg S/(g_1,\ldots,g_{n-1}) = 2 +\sum_{j=1}^{m}{s_{(j)}}$. Applying Lemma \ref{Lemma1.Reg}(a) to exact sequence (\ref{eq:E1}) it follows that $\reg S/((g_1,\ldots,g_{n-1}): g_{n})= \sum_{j=1}^{m}{s_{(j)}}$. Since $g_1,\ldots,g_{n}$ form a $d$-sequence, one has $(g_1,\ldots,g_{n-1}): g_{n}^{s} = (g_1,\ldots,g_{n-1}): g_{n}$. By applying Lemma \ref{Lemma1.Reg}(c) to the short exact sequence (\ref{eq:E4}) yields   $\reg {S}/{((g_1,\ldots,g_{n-1})+g_{n}^{s})} = 2s  +\sum_{j=1}^{m}{s_{(j)}}-1$, for all $s \geq 2$, as desired.
\end{proof}

\begin{remark}
      Note that the graphs were obtained by adding an edge $e$ between two internal vertices of a path $P_n$ such that the girth of $G$ is odd are a particular case of graphs discussed in Theorem \ref{Thm5.PPU2}. Hence we obtained the result \cite[Theorem 5.9]{SZ}.
\end{remark}

\subsection*{Conclusion} \begin{proof}[Proof of Theorem~\ref{U1.1}]
The statement $(i)$ follows from Theorems \ref{Thm4.UC1} and \ref{Thm5.PPUC1}, and Remark \ref{rem3.iG}(a), and the statement $(ii)$ follows from Theorems \ref{Thm4.U2} and \ref{Thm5.PPU2}, and Remark \ref{rem3.iG}(b).
\end{proof}

\end{document}